%% file: main.tex
\documentclass{article}
\usepackage{times,pifont,graphicx,color,amsmath,amsfonts}
\usepackage{amssymb,amsthm,url}
\definecolor{gray}{gray}{.8}

\date{}
\usepackage[ruled,vlined,linesnumbered,algonl]{algorithm2e}
\usepackage{bbding}
\usepackage[table,dvipsnames]{xcolor}
\usepackage{diagbox,caption,newclude}

\makeatletter
\newdimen\@myBoxHeight%
\newdimen\@myBoxDepth%
\newdimen\@myBoxWidth%
\newdimen\@myBoxSize%
\DeclareRobustCommand{\SquareBox}[3]{%
	 \settoheight{\@myBoxHeight}{#3}
    \settodepth{\@myBoxDepth}{#3}
    \settowidth{\@myBoxWidth}{#3}
    \pgfmathsetlength{\@myBoxSize}{max(\@myBoxWidth,(\@myBoxHeight+\@myBoxDepth))}%
	\protect\tikz \node [shape=rectangle, color=white, text=white, minimum size=\@myBoxSize,
    			 fill=#2, inner sep=0, outer sep=0pt] 
    			 {{\fontfamily{phv}\selectfont \normalsize \textcolor{white}{#1}}};%
}%
\makeatother

\newcommand\biog{\list{}{\topsep18pt plus2pt minus2pt
	\leftmargin0pt\labelwidth0pt\labelsep.3em
	\parsep0pt\listparindent\parindent
	\itemsep18pt plus2pt minus2pt
	\def\makelabel##1{\hskip\labelsep\bfseries\MakeUppercase{##1}}}%
	\let\affil\biogaffil\let\endaffil\endbiogaffil\footnotesize}

\def\affil{\list{}{\partopsep0pt\topsep12pt plus2pt minus2pt
	\leftmargin0pt\labelwidth0pt\labelsep0pt
	\parsep0pt\listparindent0pt
	\itemsep\medskipamount
	\parsep0pt
	}\footnotesize
	\item[]\itshape}
\let\endaffil\endlist

\def\biogaffil{\list{}{\partopsep0pt\topsep0pt
	\leftmargin0pt\labelwidth0pt\labelsep0pt
	\parsep0pt\listparindent0pt
	\itemsep\medskipamount
	\parsep0pt
	}\footnotesize
	\item[]\itshape}
\let\endbiogaffil\endlist

\theoremstyle{theorem}
\newtheorem{theorem}{Theorem}
\newtheorem{lemma}[theorem]{Lemma}
\newtheorem{corollary}[theorem]{Corollary}
 
\theoremstyle{definition}

\newtheorem{remark}[theorem]{Remark}
\newtheorem{proposition}[theorem]{Proposition}
\newtheorem{strategy}[theorem]{Strategy}

\usepackage{graphicx}
\usepackage{tikz}
\usetikzlibrary{arrows}
\usepackage[pdfencoding=auto, breaklinks, psdextra, colorlinks]{hyperref}

\definecolor{wgray}{RGB}{30,30,30}
\definecolor{wgreen}{RGB}{106,170,100}
\definecolor{wyellow}{RGB}{201,180,88}

\usetikzlibrary{calc}
\usetikzlibrary{shapes}
\usepackage{xcolor}
\DeclareRobustCommand{\wordlify}[2]{
	{\protect
	\hspace{-15pt}
	\raisebox{-5pt}{
		\foreach \l [count=\c, evaluate=\c as \co using {#2[\c-1]}] in #1 {
			\ifnum \co=0
				\SquareBox{\l}{wgray}{3pt}
			\else
				\ifnum \co=1
					\SquareBox{\l}{wyellow}{3pt}%
				\else
					\SquareBox{\l}{wgreen}{3pt}%
				\fi
			\fi
			\hspace{-4pt}
		}
	}
	}
}

\title{Playing Mastermind with \\Wordle-like Feedback}
\author{Renyuan Li\thanks{Department of Mathematics, Nanjing University, ryanlry9@gmail.com},  Shenglong Zhu\thanks{Department of Mathematics, Nanjing University, 502023210036@smail.nju.edu.cn}
  }

\begin{document}

\maketitle

\begin{abstract}
	We introduce an extension of Mastermind called Clear Mastermind with enhanced feedback inspired by that from Wordle. 
	The only difference between Clear Mastermind and Mastermind is a rule that provides more precise feedback, as found in Wordle. In Clear Mastermind, the feedback contains the positions of the colors the codebreaker guessed correctly and the positions of colors that appear in the answer but in different positions. 
	We explore the fewest number of guesses that a codebreaker requires to find the answer in Clear Mastermind according to its two parameters: the number of colors and the length of the answer.
\end{abstract}
\include*{chapters}

\begin{acknowledgment}{Acknowledgment.}
    The authors would like to thank their supervisor, Professor Liang Yu, for his insightful suggestions and encouragement. He pointed out some mistakes in a previous version of the paper and guided the authors to find another elegant and accurate proof. The authors also extend their sincere thanks to the anonymous referees for their constructive suggestions that led to a significant improvement of the paper.
\end{acknowledgment}

\bibliographystyle{abbrv}
\bibliography{reference}

\vfill\eject

\end{document}

%% file: chapters.tex
\section*{Introduction. }
\subsection{Background}
\textit{Mastermind} is a popular code-breaking game for two players. One player, the codemaker, 
selects a secret color sequence of four pegs from six possible colors.
The other player, known as the codebreaker, attempts to deduce the codemaker's sequence with minimal 
guesses by utilizing corresponding feedback.
Each guess is a sequence of four colored pegs picked from the six color alternatives.
The feedback includes the number of pegs the codebreaker guessed correctly and the number of pegs that are of the right color yet not in the right position.
The game has been studied from a computational perspective \cite{doerr2016playing,el2018query,larcher2021solving,martinsson2020mastermind}.

\textit{Wordle} is a single-player code-breaking game where the player tries to guess a five-letter English word chosen by the computer.
It is similar to Mastermind where the player tries to guess a color combination.
In Wordle, however, the player must guess a valid five-letter word in English for each guess.
The feedback in Wordle is more specific than in Mastermind, as it provides the exact positions of 
correctly guessed letters and the letters that appear in the answer but in different positions.
The online game has gained immense popularity and many studies have explored different ways to
interpret the game  \cite{bertsimas2022exact,short2022winning}.
It is worth noting that the game has been proven to be NP-complete \cite{lokshtanov_et_al:LIPIcs.FUN.2022.19,rosenbaum2022finding}. 
Specifically, the possible target words and allowable guesses in Wordle are restricted sets. The game we define in this article is a special case of Wordle in which target words and allowable guesses contain all possible sequences of $k$ characters. 

Additionally, Hamkins \cite{hamkins2022infinite} has studied the infinite number of variations of both Wordle and Mastermind, 
with some of his ideas overlapping with the approaches used in this article.
While Hamkins' work is more related to axiomatic set theory, this article focuses  on combinatorial problems.

In this paper, we introduce \textit{Clear Mastermind}, which combines aspects of both Mastermind and Wordle. 
Conceptually, Clear Mastermind is Mastermind with Wordle's feedback. 

We define \textit{Clear Mastermind} with $n\ge1$ \textit{positions} and $k\ge1$ \textit{letters} representing the colors as a two-player game as follows. 
Let us denote by $[k]$ the set $\{1,2,\dots,k\}$ of positive integers.
  One player, the codemaker, initially selects a secret answer ${y =(y_1,y_2,\dots, y_n) \in [k]^n}$ in letters from an alphabet $[k]$  
  to be the answer in the game,
    with $y_i$ being the letter on the $i$th position of the answer ($i$th letter of the answer for short).
    The other player, the codebreaker, attempts to find the answer by guessing a series of words.
	 In each round, the codebreaker guesses a word ${x = (x_1,x_2,\dots, x_n) \in [k]^n}$.
	 The codemaker then provides feedback by coloring each letter of the word either Green, Yellow, or Black 
	 based on how each letter matches the secret answer. 
	 The feedback provides the codebreaker with the commonalities between the secret answer $y$ and the guessed word $x$ in each round,
	 allowing them to narrow down the list of words possible to be the answers and make better guesses in the following rounds.
	  The feedback rules will be further explained in the following section. 
	  If the codebreaker's guess matches the answer in the $i$th round, we say that the codebreaker finds the answer 
	  in $i$ rounds or that the codebreaker solves the puzzle in $i$ rounds.

\subsection{Feedback Rules}
First, any letter of $x$ that matches with a letter of $y$ such that $x_i = y_i$ is colored Green. Next, consider the unmatched letters
in $x$ and $y$. For each $j \in [k]$, if $j$ has $n_j$ unmatched occurrences in $y$ and $m_j$ in $x$, then the first 
$\min(n_j,m_j)$ unmatched occurrences of $j$ in $x$ are colored Yellow. All remaining entries of $x$ are colored Black.
Two examples of Clear Mastermind are illustrated below.
\begin{center}
    \includegraphics[]{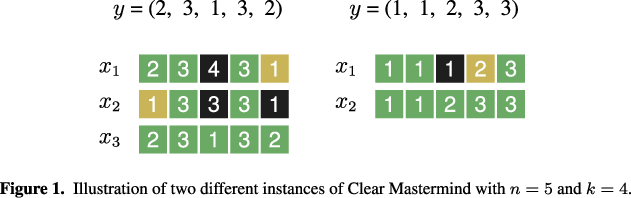}

\end{center}


			

For convenience,  we denote the feedback by a sequence made up of Green(G), Yellow(Y), and Black(B).
Let $\pi(x\mid y) \in \{\text{G,Y,B}\}^n$ denote the feedback for the guess $x$ given the answer $y$.
Define the \textit{answer set} $P_i$
\begin{align*}
    &P_0 = [k]^n,\\
    &P_i = \big\{ w \in P_{i-1}: \pi(x^{(i)}\mid w) = \widetilde{\pi}^{(i)} \big\}, i \in \mathbb{N},
\end{align*}
where $x^{(i)}$ and $\widetilde{\pi}^{(i)}$ is the guess and the feedback of the $i$th round respectively.

The answer set $P_i$ is a function of the previous guesses and feedback.
Intuitively, the answer set $P_i$ is the set of words that are still possible to be the secret answers for the codebreaker after $i$ rounds. 
By definition, we have $P_{i+1} \subset P_i, \text{ where } {i \in \mathbb{N}}$. \\
Here is an example to explain the answer set:
\begin{itemize}
	\item For $k = 3,n = 3$, the codebreaker guesses (1,2,3) in the first round and receives the feedback $\pi = $ (G,G,B), 
	then the answer set $P_1$ would be \{(1,2,1),(1,2,2)\}. In the second round, the codebreaker guesses (1,2,2) and receives the 
	feedback $\pi = $ (G,G,B), then the answer set $P_2$ would be \{(1,2,1)\}, which means that the answer is (1,2,1). The codebreaker finds the answer in 3 rounds then.
\end{itemize}
\subsection{Our Contribution}
Define $F(k,n)$ as the fewest number of guesses that the codebreaker requires to find the answer in Clear Mastermind with $n$ positions and $k$ letters.

In Clear Mastermind, the codebreaker can guess by using the following strategy. In the $j$th ($j<k$) round, 
the codebreaker guesses ($j,j,\dots, j$)$\in [k]^n$. After the codebreaker receives the feedback in the $(k-1)$th round, 
the answer set $P_{k-1}$ contains only one word and the correct answer is deduced. So, the codebreaker can solve the puzzle in the $k$th round, 
which means$\text{ for all } k,n \in \mathbb{N},$ 
\begin{equation}
	F(k,n)\le k \tag{1}
	\label{equation:1}
\end{equation}
In this work, we prove the following results:
\begin{theorem}
\label{thm:1}
\begin{enumerate}
	\item If $k\le n $,\[ F(k,n) = k.\]
	\item If $k > n$ and $n$ divides $k$,\[F(k,n) = \dfrac{k}{n}+n-1. \]
	\item If $k>n$ and $n$ does not divide $k$, \[ \left\lfloor \frac{k}{n} \right\rfloor +n-1\le F(k,n)\le \left\lfloor \frac{k}{n} \right\rfloor +n.\]
\end{enumerate}
\end{theorem}
The main results show that finding an optimal strategy in Clear Mastermind is nearly trivial, while in Wordle where guesses and target words are restricted, finding an optimal strategy is computationally intractable \cite{lokshtanov_et_al:LIPIcs.FUN.2022.19,rosenbaum2022finding}.

We prove the three cases in the following three sections.

\section{The First Case.}
\label{chap:firstcase}
In the case $k\le n$ , in order to prove Theorem~\ref{thm:1}, it is sufficient to prove $F(k,n)\ge k$ from the inequality~\eqref{equation:1}.

It is challenging to deal with feedback that includes Yellow because it provides information that is hard to utilize. 
One intuitive idea is that if we know the frequency of each letter $j \in [k]$ in the secret answer $y$, then the feedback of Yellow is unnecessary. The formal expression is discussed later in Proposition~\ref{prop:eqivalent}.
Define the feedback 
\[\tau(x\mid y) := \big( \tau_1(x\mid y),\tau_2(x\mid y),\dots,\tau_n(x\mid y) \big ) \in \{\text{\Checkmark,\XSolidBrush} \}^n\]
with $\tau_i(x\mid y) = $ \Checkmark if $x_i = y_i$ and $\tau_i(x\mid y) = $ \XSolidBrush if $x_i \ne y_i$.
We show that the feedback $\pi(x\mid y)$ is equivalent to the feedback $\tau(x\mid y)$ if we know the frequency 
of the letters in $y$.

Given a word $y$, let $\omega(y) = \big(\omega_1(y),\omega_2(y),\dots, \omega_k(y) \big)\in \mathbb{N}^k$ be the \textit{frequency vector} of $y$, 
where $\omega_j(y)$ is the number of occurrences of letter $j$ in $y$.

\begin{proposition}
\label{prop:eqivalent}
	If the codebreaker guesses a word $x$, and knows the frequency vector of the secret answer $y$, then the codebreaker can deduce $\tau(x\mid y)$ based on the provided $\pi(x\mid y)$, and vice versa, the codebreaker can deduce $\pi(x\mid y)$ given $\tau(x\mid y)$.
\end{proposition}

\begin{proof}
	On the one hand, it is easy to obtain $\tau(x\mid y)$ when the codebreaker knows $\pi(x\mid y)$. Convert every Green in $\pi(x\mid y)$ to  \Checkmark in $\tau(x\mid y)$
	and every Yellow and every Black to \XSolidBrush, and the codebreaker obtains the feedback $\tau(x\mid y)$.

	On the other hand, given $\tau(x\mid y)$, the codebreaker may obtain $\pi(x\mid y)$ as follows. Convert every \Checkmark in $\tau(x\mid y)$ to  Green in $\pi(x\mid y)$. 
	Next, for each $j \in [k]$, let $l_j$ count the number of Green occurrences of $j$ in $x$. 
	Color the leftmost $\min(\omega_j(y),\omega_j(x)-l_j)$ occurrences of $j$ in $x$ Yellow.
	Color the remaining entries of $x$ black. The codebreaker obtains the correct feedback $\pi(x\mid y)$.
\end{proof}
Thus, in the case that $n = k$ and the answer is known as a permutation of $\{1,2,\dots,k\}$, we consider the game that the codemaker alternatively provides the codebreaker with feedback $\tau(x\mid y)$, which has the same difficulty with the one that the codemaker provides the feedback $\pi(x\mid y)$.
Define $G(k)$ as the fewest number of guesses that a codebreaker requires to find the answer 
in the case that $n = k$ and the answer is known to be a permutation of $\{1,2,\dots,k\}$, while receiving the feedback $\tau(x\mid y)$ from the codemaker.

In the case $n = k$, we know that
 $$F(k,k)\ge G(k)$$
because in the game with $k$ positions, $k$ letters, and the public information that the answer is a permutation of $\{1,2,\dots,k\}$, the codebreaker can pretend not to know the public information and then solve it within $F(k,k)$ rounds.

In the case $k < n$, let $y = (y_1,y_2,\dots, y_n)$ be the answer for a game with $k$ positions and $n$ letters. 
We, as the codemaker, set $y_j = 1 $ if $j > k$ and set $(y_1,y_2,\dots, y_k)$ to be a permutation of $\{1,2,\dots,k\}$. 
Make the information above public. Under such circumstances, 
the codebreaker still needs at least $G(k)$ rounds to solve the game in the case $k<n$, which means
\[F(k,n)\ge G(k), k<n.\]
Thus, the first case of Theorem~\ref{thm:1} can be deduced from the following proposition.

\begin{proposition}
\label{prop:1}
$G(k)\ge k, \text{ for all } k \in \mathbb{N}^*$.
\end{proposition} 
\label{sec:21}
In the remaining content of the section, we will always focus on the 
Clear Mastermind game with $k$ positions, $k$ letters, and the public information that the answer is a permutation of $\{1,2,\dots,k\}$.
The codemaker provides the codebreaker with feedback $\tau(x \mid y)$.

For any codebreaker, we, as the codemaker, aim to find a word to be the answer and prove that the codebreaker 
cannot solve the puzzle in $k-1$ rounds.  In fact, we do not have to commit to a definite answer of the game in advance. 
As the game proceeds, we can choose another word as the answer of the game, as long as it remains consistent with all the 
previous guesses and the feedback. 

As a result, the codemaker only needs to provide feedback by using a specific strategy after each round and ensure that the answer set 
is not empty. The less information from the feedback that is given to the codebreaker, the better the strategy. 
Intuitively, the specific strategy is to provide \XSolidBrush \ by a greedy algorithm. It checks whether the next position of the feedback can be filled with \XSolidBrush without contradictions. If it can, fill with \XSolidBrush and go on to the next position; otherwise, fill with \Checkmark and go on.

Before providing a more detailed explanation, we extend the definition of the answer set. 
Define the \textit{Q-answer set} $Q_{i,j}$ 
\begin{align*}
    &Q_{1,0} = [k]^n,\\
    &Q_{i,0} = Q_{i-1,n}, i\in \{2,\dots,n\},\\
    &Q_{i,j} = \big\{w \in Q_{i,j-1}: \tau_j(x^{(i)}\mid w) = \widetilde{\tau_j}^{(i)}\big\},i \in [n],j\in[n],
\end{align*}
where $x^{(i)}$ and $\widetilde{\tau}^{(i)} = (\widetilde{\tau_1}^{(i)},\widetilde{\tau_2}^{(i)},\dots,\widetilde{\tau_n}^{(i)})$ are the guess and the feedback of the $i$th round respectively.

Intuitively, $Q_{i,j}$ is the set of words possible to be answers considering the guesses and
 the feedback in the initial $i-1$ rounds, and the guesses and the feedback in the initial $j$ positions of the $i$th round. 
 Notice that $P_i = Q_{i,n} = Q_{i+1,0}$. An example is shown below:
\begin{itemize}
	\item In a Mastermind game with $k = 3,n = 3$, the codebreaker guesses (1,1,1) in the first round, and the feedback is (\XSolidBrush,\Checkmark,\XSolidBrush). In the second round, the codebreaker guesses (2,1,2), and if the feedback is (\Checkmark,\Checkmark,\XSolidBrush), then the results are
\begin{align*}
&P_1= Q_{1,3} =  \big\{(2,1,2),(2,1,3),(3,1,2),(3,1,3)\big\}.\\
&Q_{2,1} = Q_{2,2} =  \big\{(2,1,2),(2,1,3)\big\}. \\
&Q_{2,3} = P_2 = \big\{(2,1,3)\big\} .
\end{align*}
\end{itemize}
 Now, we describe the specific strategy.



\begin{strategy}
\label{str}
We, the codemaker, provide the feedback $\tau = (\tau_1,\tau_2,\ldots,\tau_n)$ by the following method 
after receiving the guess $x$ from the codebreaker in the $i$ th round: 

For $\tau_1$, we search if there exists a word $a = (a_1,a_2,\ldots,a_n)\in P_{i-1}$ such that ${a_1 \ne x_1}$. If such an $a$ exists, 
set $\tau_1 = $ \XSolidBrush
and update the last answer set $Q_{i,1}$ to be those in $ P_{i-1}$ whose first letter is not $x_1$. 
$Q_{i,1}$ is the possible answer set currently.
Else, for every word $a \in P_{i-1}$, we have $a_1 = x_1$. Set $\tau_1 = $ \Checkmark
and the lastest answer set $Q_{i,1}$ equals $P_{i-1}$. 

Similarly, for $\tau_j, j\in \{2,3,\ldots,n\}$, 
we search to see if there exists word $a\in Q_{i,j-1}$ such that $a_j \ne x_j$. If such a word exists, set $\tau_j =$  \XSolidBrush
and update the lastest answer set \[Q_{i,j} = \big\{a\in Q_{i,j-1}: a_j \ne x_j \big\}.\] 
Else, set $\tau_j = $ \Checkmark 
and update the lastest answer set \[Q_{i,j} = \big\{a\in Q_{i,j-1}: a_j = x_j \big\} = Q_{i,j-1}.\]
The feedback $\tau = (\tau_1,\tau_2,\dots,\tau_n)$ is produced.
\end{strategy}



\begin{proposition}
\label{prop:2}
If the codemaker provides feedback by using Strategy~\ref{str}, 
the answer set will contain more than one word after $k-2$ rounds, i.e. $|P_{k-2}|>1$.
\end{proposition}
\begin{remark}
\label{remark:1}
It can be inferred from Proposition~\ref{prop:2} that there are at least two words to choose in the $(k-1)$st round, which means that any codebreaker cannot solve the puzzle in $k-1$ rounds facing Strategy~\ref{str}. Then we have $G(k)\ge k$.
Consequently, Proposition~\ref{prop:1} can be deduced from Proposition~\ref{prop:2}.
\end{remark}

We are now going to prove Proposition~\ref{prop:2}. We first do some preparations.

We use a 0-1 matrix $A^{(r)} = (a^{(r)}_{ji})_{k\times k}$ to track the codebreaker's guesses and the codemaker's feedback 
(tracking matrix for short) in the game after $r$ rounds. It encodes the set $P_r$, as we will prove later, and is easier to handle.

We define the matrix recursively. Initially, the values in the matrix $A^{(0)}$ are all 1, which means $a^{(0)}_{ji} = 1, \text{ for all } i,j \in [k]$.
In the $r$th round, the codebreaker provides a guess $x = (x_1,x_2,\dots, x_k) \in [k]^k$. Let us denote the feedback 
by $\tau = (\tau_1,\tau_2,\dots, \tau_k) \in\{\text{\Checkmark , \XSolidBrush}\}^k$. 
For each $i\in [k]$, if $\tau_i = $ \XSolidBrush , then set $a^{(r)}_{x_i,i}$ as 0. 
Other elements in $A^{(r)}$ remain the same as the corresponding elements in $A^{(r-1)}$.

For example, when $ k = n = 4$, a codebreaker guessed (4,2,1,3), and the feedback is (\Checkmark,\Checkmark,\XSolidBrush,\XSolidBrush); the matrix $A^{(1)}$ is then

\begin{center}

\begin{tabular}{|l|c|c|c|c|}
\hline
\rowcolor[gray]{.9}
\diagbox{letter}{position}&1&2&3&4\\
\hline 1&1&1&0&1\\
\hline 2&1&1&1&1\\
\hline 3&1&1&1&0\\
\hline 4&1&1&1&1\\
\hline
\end{tabular}
\end{center}
.We introduce the connections between $A^{(r)}$ and $P_r$.
\begin{lemma}
	\label{lemma:2}
	If the codemaker provides feedback using Strategy~\ref{str}, then for each permutation $\sigma\in S_k$, for all $r<k$,
	the word $y_0 = \big(\sigma(1),\sigma(2),\dots,\sigma(k)\big)$ exists in the answer set $P_r$ if and only if
	\[a^{(r)}_{\sigma(1),1} = a^{(r)}_{\sigma(2),2} = \dots = a^{(r)}_{\sigma(k),k} = 1,\]
	where $S_k$ is a symmetric group over all permutations of the numbers $1,2,\dots,k$, $\sigma(i)$ is the $i$th value in $\sigma$, and $A^{(r)} = (a^{(r)}_{ji})_{k\times k}$.
\end{lemma}
	
\begin{proof}
	It is trivial to see that every word $y = \big(\sigma(1),\sigma(2),\dots,\sigma(k)\big)$ in the answer set $P_r$ 
	maps to a permutation $\sigma$ in $S_k$ 
	that satisfies $a^{(r)}_{\sigma(i),i} = 1, \text{ for all } i \in [k]$.
	The rest is to prove the converse by using Strategy~\ref{str}.
	
	If $\sigma \in S_k, r \in \mathbb{N} $ satisfies that $a^{(r)}_{\sigma(i),i} = 1, i = 1,2,\dots,k$, 
	we are going to prove the word $y_0 = \big(\sigma(1),\sigma(2),\dots,\sigma(k)\big)$ is in the answer set $P_r$. 
	The answer set $P_0 $ contains the word $y_0$ because the codemaker has not guessed.
	
	We provide the proof by contradiction. If $y_0 \notin P_r$, we assume that $y_0$ is removed from the answer set 
	in the $m$th round ($m\le r$), which means $y_0 \in P_{m-1}$ and $y_0 \notin P_{m}$.
	
	Let the guess and the feedback in the $m$th round be $x = (x_1,x_2,\dots,x_n)$ and 
	${\tau = (\tau_1,\tau_2,\dots,\tau_k)}$, respectively.
	Since \[a^{(r)}_{\sigma(i),i} = 1, i = 1,2,\dots,k,\]
	we know that 
	\begin{equation}
	a^{(m)}_{\sigma(i),i} = 1, i = 1,2,\dots,k .\label{equ:1} \tag{2}
	\end{equation}
	
	Let $S = \{i\in [k]:  x_i = \sigma(i)\}$, then from Equation~\eqref{equ:1}, $\tau_i =$ \Checkmark, $\text{ for all } i \in S$. 
	For all $i \in [k]-S$, $\tau_i $ is \XSolidBrush according to Strategy~\ref{str}. 
	We know $y_0$ is still possible to be the answer word, which means $y_0 \in~ P_m$. Contradiction! Thus $y_0 \in P_r$.
\end{proof}
Intuitively, the lemma claims that $a_{j,i}^{(r)} = 1$ means that $y_i = j$ has not been eliminated as a possible solution after $r$ rounds. Thus, the product $\prod_{i=1}^{k} a_{\sigma(i),i}$ is 1 means that the permutation $\sigma = \big(\sigma(1),\sigma(2),\dots,\sigma(k)\big)$ is consistent with all responses so far and it is still possible to be the answer.
We introduce the concept of the pernament of a matrix to better explain our proof.

The \textit{permanent} of a square 0-1 matrix $A = (a_{ji})_{k\times k}$ is defined as
\[ \operatorname{perm}(A)=\sum_{\sigma \in S_{k}} \prod_{i=1}^{k} a_{\sigma(i),i}. \]
In fact, if we define $\mathcal{A}_i = \{j :  a_{ji} = 1\}$, the permutation of $A$ is the number of 
\textit{systems of distinct representatives} of $\mathcal{A}_1,\mathcal{A}_2,\dots, \mathcal{A}_k$ in the language of combinatorics.

The permanent counts the number of permutations $\sigma$ such that the corresponding product $\prod_{i=1}^{k} a_{\sigma(i),i}$ equals 1.
It can be deduced that the permanent of $A^{(r)}$ is the number of words in $P_r$.  
Thus, our last goal is to prove that $\operatorname{perm}(A^{(k-2)})>1$ in order to prove Proposition~\ref{prop:2}.
\begin{lemma}
\label{lemma:1}
Given a square 0-1 matrix $A$ with a dimension of $k \ (k\ge 2)$, if $A$ has a non-zero permanent and there are 
precisely two 1s in each of the columns, then the permanent of the matrix is larger than one.
\end{lemma}

Using Hall's marriage theorem \cite{hall1987representatives}, it is easy to obtain Lemma~\ref{lemma:1}. We provide another proof here.

\begin{proof}
We use induction on $k$.

If $k = 2$, trivial. 

Suppose that for $2\le k< k_0\ (k_0>2)$, the lemma holds. For $k = k_0$, either each row contains at least two 1s or 
there is at least a row with exactly one 1 because the permanent of the matrix is non-zero.
Suppose that there are at least two 1s in each row, from the condition that each column has 
precisely two 1s, we obtain that each column has exactly two 1s.

Because the permanent of the matrix is non-zero, there is a permutation $\sigma$ that satisfies $a_{\sigma(i),i} = 1 $ for each $i\in [k]$.
We define $\sigma_1$ by $\sigma_1(i) \ne \sigma(i)$ is the position of the other 1 in the $i$th column for each $i\in [k]$.
If there exists $i,j,b\in [k]$, such that $i\ne j$ and $\sigma_1(i) = \sigma_1(j) = b$, from the fact that $\sigma$ is a permutation we know that
there exists the only $l$ such that $\sigma(l) = b$. Additionally, from the definition of $\sigma_1$ we know $i \ne l$, $j\ne l$.
Then the $b$th row has three 1s on the $i$th column, the $j$th column, and the $l$th column. Contradiction! 

For all $i,j\in[k]$, if $i\ne j$, then $\sigma_1(i) \ne \sigma_1(j)$.
It implies that $\sigma_1$ is a permutation, $\operatorname{perm}(A) \ge 2$. 

Otherwise, we assume that in the $j$th row, there is only one 1. 
Assume that  $a_{ji} = 1$, then 
\[\operatorname{perm}(A) = a_{ji} \cdot \operatorname{perm}(A_0) = \operatorname{perm}(A_0),\]
where $A_0$ is the submatrix formed by deleting the $j$th row and $i$th column of $A$, 
a square 0-1 matrix with a dimension of $n-1$. By induction,  $\operatorname{perm}(A) = \operatorname{perm}(A_0) \ge 2$.
\end{proof} 

\begin{corollary}
\label{cor:1}
Given a square 0-1 matrix with a dimension of $k\ (k\ge2)$, if the matrix has a non-zero permanent 
and there are at least two 1s in each column, then the permanent of the matrix is larger than one.
\end{corollary}
Finally, we provide the proof of Proposition~\ref{prop:2}.

\begin{proof}
According to Strategy~\ref{str}, each column of the tracking matrix $A^{(j)}$ has at least $(n-j)$ 1s and has a positive permanent after $j$ rounds. Thus, after $(k-2)$ rounds, the tracking matrix $A^{(k-2)}$ satisfies the conditions in Corollary~\ref{cor:1}. 
The permanent of the matrix $A^{(k-2)}$ is larger than one, which means the answer set $P_{k-2}$ contains more than one word by Lemma~\ref{lemma:2}.
Thus, Proposition~\ref{prop:2} is proved. 
\end{proof}

By Remark~\ref{remark:1}, Proposition~\ref{prop:1} is proved. 
From the claim in the start of this section, the first case of Theorem~\ref{thm:1} holds, i.e.,
\begin{equation*}
F(k,n) = k, \text{ for all } k\le n.
\end{equation*}

\section{The Second Case.}
For the case $k> n$ and $n$ divides $k$, we are going to prove 
\[F(k,n) = \dfrac{k}{n}+n-1.  \]

Similar to the analysis in the first case, the codemaker only needs to provide feedback by using a specific 
strategy after each round and ensure that the answer set is not empty.
\begin{lemma}
\label{lemma:3}
For Clear Mastermind with $k$ positions, $k$ letters, and the public information that the answer is a 
permutation of $\{1,2,\dots,k\}$, if there exists $ i,j \in \{1,2,\dots,k\}$ such that the codebreaker knows 
the $i$th letter of the answer is not $j$, then the codebreaker can solve the puzzle within $k-1$ rounds.
\end{lemma}

\begin{proof}
Without loss of generality, we assume that $i = j =1$. Then, in the $r$th round when $r<k-1$, 
the codebreaker guesses $\big((r+1),r,r,\dots, r\big) \in [k]^k$ until the first position of the feedback is Green. 
If it appears, the letter $r+1$ won't appear in other positions. The codebreaker can simply guess 
${\big((r+2),(r+2),\dots,(r+2)\big)}$, ${\big((r+3),(r+3),\dots,(r+3)\big)}$ for the following rounds until the $(k-2)$th round. 
The answer set $P_{k-2}$ will contain only one word.

Otherwise, after $k-2$ rounds, the answer set $P_{k-2}$ also contains only one word. Thus Lemma~\ref{lemma:3} holds.
\end{proof}

Next we prove that when $k> n$ and $n$ divides $k$,
\[F(k,n) = \dfrac{k}{n}+n-1.\]
\begin{proof}
\label{proof:32}

On the one hand, in the initial $\frac{k}{n}-1$ rounds, the codemaker could provide feedback only containing Black 
because there are $k$ kinds of letters in total and each guess contains $n$ letters at most. 
Consequently, the best case for the codebreaker after $\frac{k}{n}-1$ rounds is equivalent to a Clear Mastermind game 
with $n$ positions and $n$ letters, 
which needs $n$ more rounds to solve. Consequently,
\[F(k,n)\ge \dfrac{k}{n}+n-1.\]
On the other hand, in the $r$th round when $r\le \frac{k}{n}$, the codebreaker could guess the word
${\big((r-1)\cdot n+1,(r-1)\cdot n+2,\dots,r\cdot n\big)}$. 
The codebreaker then knows all of the letters that appear in the answer after $\frac{k}{n}$ rounds. 

If there are fewer than $n$ kinds of letters appearing in the answer, the codebreaker needs no more than $n-1$ rounds 
to solve the puzzle since 
\[F(m,n) \le m \le n-1,\text{ for all } m<n,\]
according to the bound~\eqref{equation:1}.

Otherwise, there are $n$ kinds of letters appearing in the answer. 
From the initial $\frac{k}{n}$ guesses, we have obtained the letters appearing in the answer.  
If at least one Green appears in the feedback from the initial $\frac{k}{n}$ rounds, 
the codebreaker needs no more than $G(n-1) = n-1$ rounds to solve the puzzle. 
Otherwise, the feedback from the initial $\frac{k}{n}$ rounds has at least one Yellow.
Then by Lemma~\ref{lemma:3}, the codebreaker also needs $n-1$ rounds to solve the puzzle.
\end{proof}
Above all, we prove that 
\[F(k,n) = \dfrac{k}{n}+n-1, \text{ when }k> n \text{ and } n \text{ divides } k\]

\section{The Third Case and Future Work.}
Since
\[F(k,n)\le F(k+1,n),\quad \text{ for all } k,n \in \mathbb{N}^*,\]
and 
\[F(k,n) = \dfrac{k}{n}+n-1, \text{ when }k>n \text{ and } n \text{ divides } k\]

then the third case of Theorem~\ref{thm:1} holds, i.e.,
\[ \left\lfloor \frac{k}{n} \right\rfloor +n-1\le F(k,n)\le \left\lfloor \frac{k}{n} \right\rfloor +n, \text{ when }k>n \text{ and } n \text{ does not divide } k.\]

To obtain the exact result, we simulated the game using a computer program and obtained the results in Table~\ref{table:1} when $n<k<2n$.
\begin{table}[ht]
\begin{tabular}{|l|c|c|}
\hline
\rowcolor[gray]{.9}
\hline &$F(k,n) = n$ &$F(k,n) = n+1$ \\
\hline $n = 2$& &$F(3,2) = 3$\\
\hline $n = 3$&$F(4,3) = 3$&$F(5,3) = 4$\\
\hline $n = 4$&$F(5,4) = F(6,4)=F(7,4) = 4$&\\
\hline $n = 5$&$F(6,5) =F(7,5) =F(8,5) =F(9,5) = 5$&\\
\hline
\end{tabular}
\\
\caption{Results of Clear Mastermind calculated by the computer when ${n<k<2n}$.}
\label{table:1}
\end{table}
Denote by $\mathbf{N} $ the number of the words in the game with $k$ letters and $n$ positions, which is $k^n$.
The complexity of the program is $O (\mathbf{N}^2)$, making it challenging to compute $F(2n-1,n)$ when $n>5$.

We speculate that 
\[F(2n-1,n) = n, \text{ for all } n>3.\]
It doesn't hold for smaller $n$. We suspect that the reason is when $n$ is small, the feedback of Green provides little information about 
the other letters of the answer. For example, if the feedback contains two types of Green, representing whether the letter exists in other positions, 
we can verify that $F(3,2) = 2 \text{ and } F(5,3) = 3$. 

However, it is difficult to prove our speculation by simplifying the game as we did in the first case. 
In Section~\ref{chap:firstcase}, we demonstrated that it is sufficient to prove $F(k,n)\ge k$ when $k\le n$, 
which means we need to prove the codebreaker cannot solve the Clear Mastermind 
with $k$ letters and $n$ positions in $k-1$ rounds if $k\le n$.
Thus we can make the Clear Mastermind easier to solve by providing the codebreaker with more information, 
which also helps us to show that the codebreaker still cannot win in $k-1$ rounds. 
However, to prove $F(2n-1,n) = n, \text{ for all } n>3$, we need to 
prove $F(2n-1,n)< n+1,\text{ for all } n>3$ according to the third case of Theorem~\ref{thm:1}.
One common approach is to make the game more difficult in some way that the codebreaker can still win in $n$ rounds. 
This is more challenging than the former case.

Another potential extension that follows from our study is the relationship between $F(k+n,n)$ and $F(k,n)$ 
when $k > n$ and $n$ does not divide $k$. It is trivial that ${F(k+n,n)\ge F(k,n)+1}$. 
However, we do not know whether the speculation $F(k+n,n)\le F(k,n)+1$ is true.


\section{Conclusion.}
In this paper, we introduced Clear Mastermind, a combination of Mastermind and Wordle. 
We investigated the fewest number of guesses that a codebreaker required to find the answer in the game according to its two parameters: 
the number of letters and the length of the answer. 

In the first case where $k\le n$, we simplified the game by switching the feedback, providing more information to the codebreaker, 
proposing a feedback strategy and introducing a tracking matrix. We found the fewest number of guesses required in the end.

In the second case where $k > n$ and $n$ divides $k$, the result is mainly based on the thoughts in the first case.

In the third case where $k > n$ and $n$ does not divide $k$, the result based on the first two cases is given. However, there is still
further work to do.